\documentclass[11pt,leqno]{amsart}
\usepackage{verbatim,amssymb,amsfonts,latexsym,amsmath,amsthm,bbm,nicefrac,xspace,enumerate,tikz}

\newtheorem{theorem}{Theorem}[section]

\newtheorem{lemma}[theorem]{Lemma}
\newtheorem{corollary}[theorem]{Corollary}

\newtheorem{question}[theorem]{Question}
\newtheorem*{claim}{Claim}
\theoremstyle{definition}

\theoremstyle{remark}

\newcommand{\Q}{\mathbb{Q}}
\newcommand{\R}{\mathbb{R}}

\newcommand{\CC}{\mathbb{C}}

\newcommand{\D}{\mathcal{D}}
\newcommand{\F}{\mathcal{F}}
\newcommand{\G}{\mathcal{G}}

\renewcommand{\P}{\mathcal{P}}

\newcommand{\U}{\mathcal{U}}

\newcommand{\explicitSet}[1]{\left\lbrace #1 \right\rbrace}
\newcommand{\brackets}[1]{\left\langle #1 \right\rangle}
\newcommand{\set}[2]{\explicitSet{#1 \colon #2}}
\newcommand{\seq}[2]{\brackets{#1 \colon #2}}
\newcommand{\<}{\langle}
\renewcommand{\>}{\rangle}
\renewcommand{\a}{\alpha}
\renewcommand{\b}{\beta}
\newcommand{\g}{\gamma}
\newcommand{\dlt}{\delta}

\renewcommand{\k}{\kappa}
\newcommand{\s}{\sigma}

\newcommand{\w}{\omega}

\newcommand{\sub}{\subseteq}
\newcommand{\rest}{\!\restriction\!}
\newcommand{\cat}{\!\,^{\frown}}

\newcommand{\card}[1]{\left\lvert #1 \right\rvert}

\newcommand{\PP}{\mathbb{P}}
\newcommand{\forces}{\Vdash}

\newcommand{\BB}{\mathbb{B}}

\renewcommand{\SS}{\mathbb{S}}

\newcommand{\continuum}{\mathfrak{c}}

\newcommand{\dom}{\mathfrak d}
\newcommand{\bdd}{\mathfrak b}

\newcommand{\ch}{\ensuremath{\mathsf{CH}}\xspace}
\newcommand{\gch}{\ensuremath{\mathsf{GCH}}\xspace}
\newcommand{\zfc}{\ensuremath{\mathsf{ZFC}}\xspace}

\newcommand{\MH}{\ensuremath{\mathsf{MH}}\xspace}
\newcommand{\embeds}{<\hspace{-2.5mm} \cdot \hspace{1.5mm}}

\begin{document}

\title{On Roitman's principles $\mathsf{MH}$ and $\Delta$}
\author{Hector Barriga-Acosta}
\address {
H. Barriga-Acosta\\
Department of Mathematics and Statistics\\
University of North Carolina at Charlotte\\
Charlotte, NC 28223, USA}
\email{hbarriga@charlotte.edu}

\author{Will Brian}
\address {
W. R. Brian\\
Department of Mathematics and Statistics\\
University of North Carolina at Charlotte\\
Charlotte, NC 28223, USA}
\email{wbrian.math@gmail.com}
\urladdr{wrbrian.wordpress.com}

\author{Alan Dow}
\address {
A. S. Dow\\
Department of Mathematics and Statistics\\
University of North Carolina at Charlotte\\
Charlotte, NC 28223, USA}
\email{adow@charlotte.edu}

\subjclass[2010]{}

\thanks{The second author is supported by NSF grant DMS-2154229.}

\begin{abstract}
The Model Hypothesis (abbreviated $\mathsf{MH}$) and $\Delta$ are set-theoretic axioms introduced by J. Roitman in her work on the box product problem. Answering some questions of Roitman and Williams on these two principles, we show (1) $\mathsf{MH}$ implies the existence of $P$-points in $\omega^*$ and is therefore not a theorem of $\mathsf{ZFC}$; (2) $\mathsf{MH}$ also fails in the side-by-side Sacks models; 
(3) as $\Delta$ holds in these models, this implies $\Delta$ is strictly stronger than $\mathsf{MH}$; 
(4) furthermore, $\Delta$ holds in a large class of forcing extensions in which it was not previously known to hold.
\end{abstract}

\maketitle


\section{Introduction}

The \emph{box product problem} asks whether the countable box product $\square \R^\w$ is normal. 
First posed by Tietze in the 1940's, this is perhaps the oldest open question in general or set-theoretic topology (and it is certainly one of the oldest); see \cite{Williams}, \cite{RW}, and \cite{Roitman2} for references. 
The problem was partly solved in the 1970's when Mary Ellen Rudin proved that the Continuum Hypothesis (\ch) implies $\square \R^\w$ is paracompact, hence normal, in \cite{Rudin}. 
But whether the normality of $\square \R^\w$ is a consequence of \zfc or independent of it remains an important open question, and this is what is meant by the ``box product problem'' today.

An important variant of the box product problem asks about the space $\square (\w+1)^\w$ instead of $\square \R^\w$. 
Because $\square (\w+1)^\w$ is a closed subspace of $\square \R^\w$, the normality of $\square \R^\w$ implies the normality of $\square (\w+1)^\w$, and in particular Rudin's work shows that $\square (\w+1)^\w$ is normal under \ch. But whether $\square (\w+1)^\w$ is consistently non-normal is an open problem. 
This version of the box product problem rose to prominence in the 1970's, in large part because of a result of Kunen proved in \cite{Kunen}:
\emph{If $X_n$ is a compact metric space for each $n \in \w$, then the box product $\square_{n \in \w}X_n$ is paracompact if and only if the nabla product $\nabla_{n \in \w}X_n$ is paracompact.} 
Recall that the \emph{nabla product} $\nabla_{n \in \w}X_n$ is defined as the quotient $\square_{n \in \w}X_n / \!=^*$ of the box product by the ``almost equal'' relation: i.e., the relation defined by taking $x =^* y$ if and only if $\set{n \in \w}{x(n) \neq y(n)}$ is finite. 
In particular, $\square (\w+1)^\w$ is paracompact if and only if $\nabla (\w+1)^\w$ is.

Many set-theoretic principles are known to imply $\square (\w+1)^\w$ is paracompact: \ch (Rudin, \cite{Rudin}), $\dom = \w_1$ (Williams, \cite{Williams}), $\bdd = \dom$ (van Douwen, \cite{vanDouwen}), $\dom = \continuum$ (Roitman, \cite{Roitman1}), and the axioms $\MH$ and $\Delta$ (Roitman, \cite{Roitman2}). 
These last two axioms, \MH and $\Delta$, were introduced by Roitman in \cite{Roitman2} (see also \cite{Roitman0,Roitman1}) specifically to deal with the box product problem, or rather its variation for $\square (\w+1)^\w$. 
These axioms represent the minimal sufficient set-theoretic tools required to push through certain arguments showing the paracompactness of $\nabla (\w+1)^\w$ (hence the paracompactness of $\square (\w+1)^\w$, by Kunen's result). 
In other words, they are very weak assumptions, seemingly the weakest possible that still enable us to prove the normality of $\square (\w+1)^\w$. 
In fact, the first author proved in joint work with Paul Gartside (see \cite{BAG}) that $\Delta$ is equivalent to the paracompactness of $\nabla (\w+1)^\w$.

Roitman left open the question of whether the principles \MH and $\Delta$ might simply follow from \zfc. If so, this would answer the $\square (\w+1)^\w$ variant of the box product problem once and for all. 

Another important question in set-theoretic topology and combinatorial set theory is whether, in the random model, there are $P$-points in $\w^*$. An alleged proof of a ``yes'' answer was published by Paul E. Cohen in \cite{Cohen} (not to be confused with Paul J. Cohen, who invented the technique of forcing), and the problem was thought to be settled for decades. 
Cohen's argument proceeds in two steps: $(1)$ he defines something called a ``pathway'' and proves that the existence of pathways implies the existence of $P$-points; $(2)$ he argues that there are pathways in the random model. 
Several years ago, Osvaldo Guzm\'an found a flaw in part $(2)$ of Cohen's argument (the flaw is explained in the appendix to \cite{CG}). Thus the problem of whether there are $P$-points in the random model was reopened, and it remains unsolved. 
However, part $(1)$ of Cohen's argument is correct: the existence of pathways implies the existence of $P$-points. 

The first theorem of this paper observes a connection between these two important open problems: 
\begin{enumerate}
\item $\mathsf{MH}$ implies the existence of pathways. Consequently, \MH implies there are $P$-points in $\w^*$, and \MH is not a theorem of $\mathsf{ZFC}$.
\end{enumerate}
All known models without $P$-points satisfy $\bdd = \dom = \aleph_1$, and therefore also satisfy $\Delta$. 
Thus, as a fairly immediate consequence of $(1)$, we deduce: 
\begin{itemize}
\item[(2)] $\Delta$ does not imply \MH.
\end{itemize}
This answers a question of Roitman from \cite{Roitman2}. Furthermore, we prove:
\begin{itemize}
\item[(3)] Pathways do not exist in the side-by-side Sacks models (so \MH fails there too). Consequently, the existence of pathways is not equivalent to the existence of $P$-points (and neither is \MH).
\end{itemize}
All these results are proved in Section~\ref{sec:P}. 
In Section~\ref{sec:ccc}, we show: 
\begin{itemize}
\item[(4)] Pathways exist in a large class of forcing extensions.
\end{itemize}
What precisely this last result means is explained in Section~\ref{sec:ccc}. For now, let us say only that the result significantly enlarges the variety of models in which $\Delta$ is known to hold. In particular, the results of Section~\ref{sec:ccc} answer a question of Roitman and Williams by showing that forcing to add $\aleph_2$ Cohen reals and then $\aleph_3$ random reals over a model of \ch produces a model of $\Delta$. 

\section{\MH and $\Delta$}

In this section we define \MH and $\Delta$ and lay out some of what is already known about them. 

Recall that $H(\k)$ denotes the set of all sets hereditarily smaller than $\k$; in particular, $H(\w_1)$ is the set of hereditarily countable sets. 
Roitman's Model Hypothesis \MH states:
\begin{itemize}
\item[\MH:$\ $] For some cardinal $\k$, there is an increasing sequence $\seq{M_\a}{\a < \k}$ of elementary submodels of $H(\w_1)$, with $H(\w_1) = \bigcup_{\a < \k} M_\a$, such that for every $\a < \k$, there is some $f \in M_{\a+1} \cap \w^\w$ such that $f \not\leq^* g$ for all $g \in M_\a \cap \w^\w$.
\end{itemize}

Observe that if \MH holds, then the $\k$ from this definition must satisfy $\bdd \leq \k \leq \dom$.
For the first inequality, note that if $\bdd > \k$ then we could find a function $g \geq^* f_\a$ for every $\a < \k$, contradicting the fact that $g \in \bigcup_{\a < \k}M_\a$ and $f_{\a+1}$ is not dominated by any function in $M_\a$. 
For the second inequality, if $\D$ is a dominating family, then we cannot have $\D \sub H_\a$ for any particular $\a$, but this situation would certainly arise if $\k > \dom$. 

Observe also that $\dom = \continuum$ implies \MH. The reason is that a L\"{o}wenheim-Skolem argument gives us an increasing sequence $\seq{M_\a}{\a < \continuum}$ of elementary submodels of $H(\w_1)$ with $H(\w_1) = \bigcup_{\a < \continuum} M_\a$ such that $|M_\a| < |H(\w_1)| = \continuum$ and $M_\a \in M_{\a+1}$ for every $\a < \continuum$. But $M_\a \in M_{\a+1}$ and $|M_\a| < \dom$ implies there is some $f \in M_{\a+1} \cap \w^\w$ that is unbounded over $M_\a \cap \w^\w$.

Let $\partial \w^\w = \bigcup \set{\w^A}{A \sub \w \text{ is infinite and co-infinite}}$. (In \cite{Roitman2}, Roitman denotes this set of partial functions by $\w^{\subset \w}$). Observe that each $x \in \partial \w^\w$ corresponds to an element $\bar x \in \square (\w+1)^\w$, namely
$$\bar x(x) \,=\, \begin{cases}
x(n) &\text{ if } n \in \mathrm{dom}(x), \\
\w &\text{ if not.}
\end{cases}$$
If $x,y \in \partial \w^\w$, then the set-theoretic difference $x \setminus y$ is a (possibly finite) partial function $\w \to \w$. If $h \in \w^\w$ and $x \in \partial \w^\w$, we define $x \not>^* h$ to mean that there are infinitely many $n \in \mathrm{dom}(x)$ such that $x(n) \leq h(n)$. 
Roitman's principle $\Delta$ states:
\begin{itemize}
\item[$\Delta$:$\ $] There exists a mapping $\partial \w^\w \to \w^\w$, which we denote by $x \mapsto h_x$, such that for any $x,y \in \partial \w^\w$, if 
\begin{enumerate}
\item $|x \setminus y| = |y \setminus x| = \aleph_0$ and
\item $\card{\set{n \in \mathrm{dom}(x) \cap \mathrm{dom}(y)}{ x(n) \neq y(n) }} < \aleph_0$, then
\end{enumerate}
then either $x \setminus y \not>^* h_y$ or $y \setminus x \not>^* h_x$.
\end{itemize}

Observe that $\bdd = \dom$ implies $\Delta$. To see this, recall that if $\bdd = \dom$ then there is a $\leq^*$-increasing enumeration $\seq{f_\a}{\a < \dom}$ of a dominating family in $\w^\w$ (i.e., a scale). Given $x \in \partial \w^\w$, define $h_x = f_\a$ where $\a$ is the minimal ordinal with $x <^* f_\a$ (by which we mean $x(n) < f_\a(n)$ for all but finitely many $n \in \mathrm{dom}(x)$). 
If $x,y \in \partial \w^\w$ satisfy (1) and (2) from the statement of $\Delta$, then either $x \setminus y \not>^* h_y$ or $y \setminus x \not>^* h_x$ depending on which of $x$ or $y$ is dominated by an earlier member of the scale.

Observe that \MH also implies $\Delta$. To see this, suppose $\seq{M_\a}{\a < \k}$ is a sequence witnessing \MH. 
By replacing each $f_{\a+1}$ from the definition of $\MH$ with a strictly larger, strictly increasing function, we may (and do) assume $x \not>^* f_{\a+1}$ for every $x \in M_\a \cap \partial \w^\w$.  
Now given $x \in \partial \w^\w$, define $h_x = f_{\a+1}$ where $\a$ is minimal with $x \in M_\a$. 
If $x,y \in \partial \w^\w$ satisfy (1) and (2) from the statement of $\Delta$, then either $x \setminus y \not>^* h_y$, if $y$ does not appear before $x$ in the sequence of $M_\a$'s, or $y \setminus x \not>^* h_x$ if $x$ does not appear before $y$.

All the above observations, plus some of the results mentioned in the introduction, are summarized in the following diagram: 

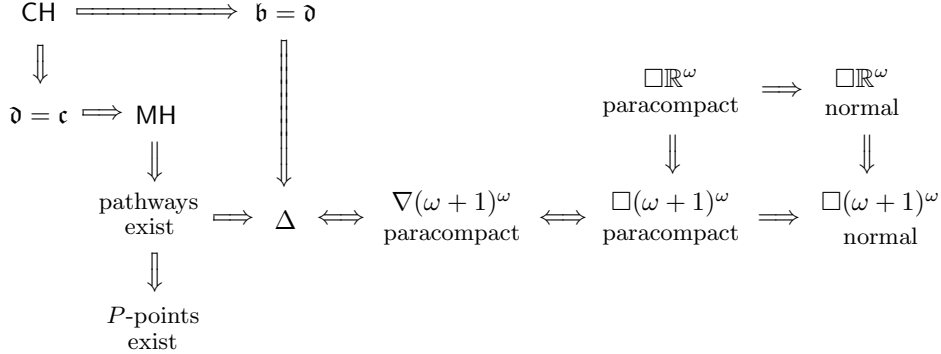
\vspace{1mm}

\begin{center}
\begin{tikzpicture}[xscale=.82,yscale=.7]

\node at (0,3.95) {\small \ch};
\node at (4,3.95) {\small $\bdd = \dom$};
\node at (0,2) {\small $\dom = \continuum$};
\node at (1.9,2) {\small $\MH$};
\node at (4,0) {\small $\Delta$};
\node at (6.7,.3) {\small $\nabla (\w+1)^\w$};
\node at (6.7,-.3) {\footnotesize paracompact};
\node at (10.25,.3) {\small $\square (\w+1)^\w$};
\node at (10.25,-.3) {\footnotesize paracompact};
\node at (13.65,.3) {\small $\square (\w+1)^\w$};
\node at (13.65,-.3) {\footnotesize normal};
\node at (1.78,.32) {\footnotesize pathways};
\node at (1.8,-.13) {\footnotesize exist};
\node at (1.85,-1.85) {\footnotesize $P$-points};
\node at (1.85,-2.3) {\footnotesize exist};
\node at (10.25,2.7) {\small $\square \R^\w$};
\node at (10.25,2.15) {\footnotesize paracompact};
\node at (13.35,2.7) {\small $\square \R^\w$};
\node at (13.35,2.15) {\footnotesize normal};

\node at (1.03,2) {\small $\implies$};
\node at (3.15,0) {\small $\implies$};
\node at (4.9,0) {\small $\Longleftrightarrow$};
\node at (8.5,0) {\small $\Longleftrightarrow$};
\node at (12,0) {\small $\implies$};
\node at (0,2.95) {\rotatebox{-90}{\small $\implies$}};
\node at (3.95,1) {\rotatebox{-90}{\small $\implies$}};
\node at (1.85,1.1) {\rotatebox{-90}{\small $\implies$}};
\node at (2.9,3.93) {\small $\implies$};
\node at (2.5,3.93) {\small $=$};
\node at (2.25,3.93) {\small $=$};
\node at (2,3.93) {\small $=$};
\node at (1.75,3.93) {\small $=$};
\node at (1.5,3.93) {\small $=$};
\node at (1.25,3.93) {\small $=$};
\node at (1,3.93) {\small $=$};
\node at (.75,3.93) {\small $=$};
\node at (1.85,-1) {\rotatebox{-90}{\small $\implies$}};
\node at (10.25,1.25) {\rotatebox{-90}{\small $\implies$}};
\node at (13.4,1.25) {\rotatebox{-90}{\small $\implies$}};
\node at (12.05,2.45) {\small $\implies$};
\node at (3.95,1.5) {\rotatebox{-90}{\small $=$}};
\node at (3.95,1.75) {\rotatebox{-90}{\small $=$}};
\node at (3.95,2) {\rotatebox{-90}{\small $=$}};
\node at (3.95,2.25) {\rotatebox{-90}{\small $=$}};
\node at (3.95,2.5) {\rotatebox{-90}{\small $=$}};
\node at (3.95,2.75) {\rotatebox{-90}{\small $=$}};
\node at (3.95,3) {\rotatebox{-90}{\small $=$}};
\node at (3.95,3.25) {\rotatebox{-90}{\small $=$}};

\draw (.719,1.98)--(.719,2.08);
\draw (2.84,-.02)--(2.84,.08);
\draw (-.02,3.312)--(.07,3.312);
\draw (3.93,3.402)--(4.02,3.402);
\draw (1.83,-.638)--(1.92,-.638);
\draw (.62,3.91)--(.62,4.01);

\end{tikzpicture}
\end{center}

Arrows that are open at the back indicate implications that are not known to reverse, and arrows that are closed at the back indicate implications that are known not to reverse. 
For example, it is open whether $\Delta$, or for that matter any of the five statements on the right side of the diagram, is a theorem of \zfc. 
We have no models in which any of these statements fail, hence no way to see that implications involving them fail to reverse. Thus all the arrows on the right side of the digram are open at the back. 

The fact that ``$P$-points exist'' does not imply ``pathways exist'' is proved in Section~\ref{sec:P}, where we show the side-by-side Sacks models have $P$-points but no pathways. (The same is true for the iterated Sacks model, by essentially the same proof, but we provide the details for the side-by-side models.) 
These same models satisfy $\Delta$, which is how we know the existence of pathways does not imply $\Delta$. We also prove in Section~\ref{sec:P} that the existence of pathways implies both $\Delta$ and the existence of $P$-points.

The fact that \ch does not imply $\bdd = \dom$ or $\dom = \continuum$ is common knowledge. So too is the fact that $\bdd < \dom = \continuum$ is consistent, and as $\dom = \continuum$ implies $\Delta$, this means $\Delta$ does not imply $\bdd = \dom$. 

To finish justifying how we have drawn our arrows, we claim that $\MH$ does not imply $\dom = \continuum$. 
To see this, 
begin with a model $V$ of $\neg \ch$, and then add a generic $G$ for the length-$\w_1$ finite support iteration of Hechler forcing. 
In the resulting model $V[G]$, we have $\dom = \aleph_1 < \continuum$, and so $\dom = \continuum$ fails. 
For each $\a < \w_1$, let $M_\a = H(\w_1)^{V[G_\a]}$ (i.e., the hereditarily countable sets in the intermediate model $V[G_\a]$ obtained by restricting $G$ to the first $\a$ stages of the iteration). 
If $V$ contains sufficiently large cardinals, then $M_\a \preceq H(\w_1)^{V[G]}$ for all $\a$, and therefore $\seq{M_\a}{\a < \w_1}$ witnesses \MH in $V[G]$. 
(For example, if there is a weakly compact Woodin cardinal, then by a result of Neeman and Zapletal in \cite{NZ} there is an elementary embedding $L(\R^{V[G_\a]}) \to L(\R^{V[G]})$ for all $\a$. 
This implies $H(\w_1)^{V[G_\a]} \preceq H(\w_1)^{V[G]}$ with room to spare.) 
Thus, unless certain large cardinal axioms turn out to be inconsistent, \MH does not imply $\dom = \continuum$.

\section{The consistent failure of \MH}\label{sec:P}

What follows is a slight generalization of the notion of a pathway defined by Paul E. Cohen in \cite{Cohen}. Given $f,g \in \w^\w$, define the \emph{join} of $f$ and $g$ to be the function $f \vee g \in \w^\w$ given by
$$(f \vee g)(n) \,=\, \begin{cases}
f(\frac{n}{2}) &\text{ if $n$ is even,} \\
g(\frac{n+1}{2}) &\text{ if $n$ is odd.}
\end{cases}$$
A \emph{pathway} is an increasing sequence $\seq{A_\a}{\a < \k}$ of subsets of $\w^\w$, for some cardinal $\k$, such that $\bigcup_{\a < \k}A_\a = \w^\w$ and, for all $\a < \k$,
\begin{itemize}
\item[$\circ$] $A_\a$ does not dominate $A_{\a+1}$,
\item[$\circ$] if $f,g \in A_\a$ then $f \vee g \in A_\a$, and
\item[$\circ$] $A_\a$ is closed under Turing reducibility.
\end{itemize} 
The reason this is a slight generalization of Cohen's definition is because Cohen requires $\k = \dom$. We omit this requirement because (1) it is not needed to prove that the existence of pathways implies the existence of $P$-points, and (2) omitting it enables us to prove the following theorem. 
Note however, that if $\seq{A_\a}{\a < \w_1}$ is a pathway, then $\bdd \leq \k \leq \dom$, for exactly the same reasons that a witness to \MH must have length between $\bdd$ and $\dom$.

\begin{theorem}
\MH implies there is a pathway.
\end{theorem}
\begin{proof}
Suppose $\seq{M_\a}{\a < \k}$ is a sequence of models witnessing \MH. For each $\a < \k$, let $A_\a = M_\a \cap \w^\w$. 
Then $\seq{A_\a}{\a < \k}$ is a sequence of subsets of $\w^\w$. It is increasing because $\seq{M_\a}{\a < \k}$ is increasing, and $\bigcup_{\a < \k}A_\a = \w^\w \cap \bigcup_{\a < \k}M_\a = \w^\w$.  
By the definition of \MH, some $f_{\a+1} \in M_{\a+1}$ is not dominated by $M_\a \cap \w^\w$; i.e., $A_\a$ does not dominate $A_{\a+1}$.

Because each $M_\a$ is an elementary substructure of $H(\w_1)$, each $M_\a$ is closed under basic set-theoretic operations (like taking the join of two functions), which means each $A_\a$ is closed under the join operator. 
Finally, if $f \in A_\a$ then $\set{g \in \w^\w}{g \leq_T f}$ (where $g \leq_T f$ means that $g$ is Turing reducible to $f$) is countable. 
If $f \in M_\a$ then $\set{g \in \w^\w}{g \leq_T f} \in M_\a$ by elementarity, because this set is a member of $H(\w_1)$ definable from $f$. 
But countable members of $M_\a$ are subsets of $M_\a$ (another consequence of elementarity), so $\set{g \in \w^\w}{g \leq_T f} \sub M_\a$. 
Hence $A_\a$ is closed under Turing reducibility, and $\seq{A_\a}{\a < \k}$ is a pathway.
\end{proof}

A notion of ``strong pathways'' is defined in \cite{FBH}, and we note in passing that the existence of these strong pathways is very similar to the assertion that there is a witness to \MH with length $\k = \w_1$. (It would be equivalent to it if the word ``elementary'' were deleted from the definition of \MH, and replaced with the weaker requirement that each $M_\a$ be a model of $\zfc^-$.) More on pathways can also be found in \cite{FB} or in the appendix to \cite{CG}.

\begin{question}
Is the existence of pathways equivalent to \MH?
\end{question}

\begin{theorem}
The existence of a pathway implies $\Delta$. 
\end{theorem}

\noindent The proof of this theorem is just a small modification of the proof that \MH implies $\Delta$ (which is due to Roitman).

\begin{proof}
Suppose $\seq{A_\a}{\a < \k}$ is a pathway. 
For each $\a < \k$, there is some function not dominated by $A_\a$, which implies there is a non-decreasing function not dominated by $A_\a$. 
Let $f_\a$ be some such function. 

For each $x \in \partial \w^\w$, define a function $\tilde x \in \w^\w$ by setting 
$$\tilde x (n) \,=\, \begin{cases}
x(n)+1 &\text{ if } n \in \mathrm{dom}(x),\\
0 &\text{ if } n \notin \mathrm{dom}(x)
\end{cases}$$
for all $n$. 
The total function $\tilde x$ is computable from $x$, and the partial function $x$ is computable from $\tilde x$. 
For each $x \in \partial \w^\w$, there is some $\a < \k$ such that $\tilde x \in A_\a$. Let $\a_x$ denote the minimum such $\a < \k$, and define $h_x = f_{\a_x}$.

Now suppose that $x,y \in \partial \w^\w$ satisfy (1) and (2) from the statement of $\Delta$: i.e., $\card{x \setminus y} = \card{y \setminus x} = \aleph_0$ and $\card{\set{n \in \mathrm{dom}(x) \cap \mathrm{dom}(y)}{x(n) \neq y(n)}} < \aleph_0$. 
Furthermore, suppose $\a_x \leq \a_y$. 

Because the $A_\a$'s are increasing, we have $\tilde x,\tilde y \in A_{\a_y}$, and so $\tilde x \vee \tilde y \in A_{\a_y}$. 
Because $x$ and $y$ are each computable from $\tilde x \vee \tilde y$, this implies that any function computable from $x$ and $y$ is a member of $A_{\a_y}$. 
In particular, $A_{\a_y}$ contains the function $g$ defined by 
$$g(n) \,=\, x(\min \set{m \geq n}{m \in \mathrm{dom}(x \setminus y)}),$$
which is well-defined because $\mathrm{dom}(x \setminus y)$ is infinite. 
Thus $h_y = f_{\a_y} \not\leq^* g$, that is, there are infinitely many $n$ such that $f_{\a_y}(n) > g(n)$. 
Fix $n$, and let $m = \min \set{m \geq n}{m \in \mathrm{dom}(x \setminus y)}$. Using the fact that $f_{\a_y}$ is non-decreasing, $f_{\a_y}(n) > g(n)$ implies
$$h_y(m) = f_{\a_y}(m) \geq f_{\a_y}(n) > g(n) = x(m)$$
and $m \in \mathrm{dom}(x \setminus y)$. 
Because there are infinitely many $n$ with $f_{\a_y}(n) > g(n)$, this implies there are infinitely many $m \in \mathrm{dom}(x \setminus y)$ with $h_y(m) > x(m)$. In other words, $x \setminus y \not>^* h_y$. Similarly, if $\a_y \leq \a_x$ then $y \setminus x \not>^* h_x$.
\end{proof}

Next we include a proof that the existence of pathways implies there are $P$-points in $\w^*$. While this can be found in Cohen's paper \cite{Cohen}, we have chosen to include a proof for three reasons. First, we have slightly expanded Cohen's definition of a pathway, so we must argue that our generalized pathways still imply the existence of $P$-points. 
Second, some readers may not trust a proof in a paper known to contain an incorrect proof. 
The third reason is simply the convenience of the reader who wishes to see the proof. 

Ketonen proved in \cite{Ketonen} that $\dom = \continuum$ implies there are $P$-points in $\w^*$. 
Given that $\dom = \continuum$ implies the existence of pathways, but not vice versa, the following theorem can be seen as strengthening Ketonen's result. 

\begin{theorem}\label{thm:Ppoints}
The existence of a pathway implies there are $P$-points in $\w^*$.
\end{theorem}
\begin{proof}
Our proof more or less follows Ketonen's. 
Fix a pathway $\seq{A_\a}{\a < \k}$. 

To begin, note that subsets of $\w$ can be ``coded'' by their characteristic functions. 
Thus, even though $A_\a$ contains functions and not sets, we may think of it as describing a collection of subsets of $\w$: 
for each $\a < \k$, define 
$$\mathrm{Set}_\a \,=\, \set{B \sub \w}{\chi_{{}_B} \in A_\a}.$$
Given $B,C \sub \w$, note that the characteristic function $\chi_{{}_{B \cap C}}$ is computable from $\chi_{{}_B} \vee \chi_{{}_C}$. Because $A_\a$ is closed under the $\vee$ operator and Turing reducability, this means that $\mathrm{Set}_\a$ is closed under binary intersections. 
It follows (via induction) that $\mathrm{Set}_\a$ is closed under finite intersections. 
Furthermore, $\mathrm{Set}_\a$ is closed under Turing reducibility, because $A_\a$ is. 
Also, if $B \in \mathrm{Set}_\a$ then the natural enumeration of $B$ (i.e., the function mapping $n$ to $\text{the $n^{\mathrm{th}}$ element of $B$}$), is in $A_\a$. 

Sequences of subsets of $\w$ can also be coded with functions. 
Fix some coding/decoding functions $c$ and $d$ such that for any sequence $\vec{s} = \seq{B_n}{n \in \w}$ of subsets of $\w$, $c(\vec s\,) \in \w^\w$ and $d(c(\vec s\,)) = \vec s$. 
Furthermore, do this in such a way that $\chi_{{}_{B_n}}$ (the characteristic function of the $n^\mathrm{th}$ member of $\vec s\,$) is uniformly (in $n$) computable from $c(\vec s\,)$ for all $n$. 
(This can be accomplished, for example, with a computable pairing function $\w \times \w \to \w$, which can be used to code a sequence of characteristic functions into a single $0$-$1$ sequence.)
For each $\a < \k$, let 
$$\mathrm{Seq}_\a \,=\, \set{\vec s \in (\P(\w))^\w}{c(\vec s\,) \in A_\a}.$$
Because the $A_\a$ are closed under Turing reducibility, every subset of $\w$ computable from some $\vec s \in \mathrm{Seq}_\a$ is in $\mathrm{Set}_\a$, and every sequence of sets computable from $\vec s$ is in $\mathrm{Seq}_\a$. 
For example, if $\vec s = \seq{B_n}{n \in \w} \in \mathrm{Seq}_\a$, then $B_n \in \mathrm{Set}_\a$ for all $n$, and
$\bigcap_{i \leq n}B_i \in \mathrm{Set}_\a$ for all $n$, and $\seq{\bigcap_{i \leq n}B_i}{n \in \w} \in \mathrm{Seq}_\a$.

\begin{claim}
Fix $\a < \k$, and suppose $\F \sub A_\a$ is a free filter on $\w$.
For every $f \in \w^\w$ not dominated by $A_\a$,
there is a function $\psi$, computable from $f$, 
such that $\psi$ maps $\mathrm{Seq}_\a \cap \F^\w$ to infinite subsets of $\w$, in such a way that
\begin{enumerate}
\item For every $\vec s \in \mathrm{Seq}_\a \cap \F^\w$, $\psi(\vec s\,)$ is a pseudo-intersection for $\vec s$.
\item $\F \cup \set{\psi(\vec s\,)}{\vec s \in \mathrm{Seq}_\a \cap \F^\w}$ is a filter base.
\end{enumerate}
\end{claim}
\vspace{1mm}
\noindent \emph{Proof of Claim.$\ $} Fix some $f \in \w^\w$ that is not dominated by $A$: i.e., $f \not\leq^* g$ for every $g \in A$.
Given $\vec s = \seq{B_n}{n \in \w} \in \mathrm{Seq}_\a \cap \F^\w$, define
$$\psi(\vec s\,) \,=\, \textstyle \bigcup_{n \in \w} \big( f(n) \cap \bigcap_{i \leq n} B_i \big).$$
For each $n \in \w$, the set $\bigcap_{i \leq n} B_i$ is infinite (because $\F$ is a free filter and $B_0,\dots,B_n \in \F$).

As mentioned above, $\vec s \in \mathrm{Seq}_\a$ implies $\seq{\bigcap_{i \leq n}B_i}{n \in \w} \in \mathrm{Seq}_\a$. 
From (the code for) this sequence, one can compute the function $g$ mapping $n \in \w$ to the $n^\mathrm{th}$ member of $\bigcap_{i \leq n}B_i$. So $g \in A_\a$. 
Hence $g \not>^* f$, which implies $f(n) \cap \bigcap_{i \leq n} B_i$ has size $\geq\! n$ for infinitely many $n$. 
Thus $\psi(\vec s\,)$ is infinite.
As $\psi(\vec s\,) \setminus B_n \sub f(n)$ for all $m$, this means $\psi(\vec s\,)$ is a pseudo-intersection for $\vec s$.

To finish the proof of the claim, it remains to check that
$$\G \,=\, \F \cup \set{\psi(\vec s\,)}{\vec s \in \mathrm{Seq}_\a \cap \F^\w}$$
is a filter base. 
To this end, let us fix some $F_0,F_1,\dots,F_{k-1} \in \F$, and some $\vec s_0 = \seq{B^0_n}{n \in \w},\vec s_1 = \seq{B^1_n}{n \in \w},\dots,\vec s_n = \seq{B^{\ell-1}_n}{n \in \w}$ in $A_\a \cap \F^\w$.
Then define a new sequence $\vec{\hspace{.4mm}t}\, = \seq{Y_n}{n \in \w}$ of subsets of $\w$ by setting
$$Y_n \,=\, F_0 \cap F_1 \cap \dots \cap F_{k-1} \cap B^0_n \cap B^1_n \cap \dots \cap B_n^{\ell-1}$$
for all $n \in \w$.
Note that $F_i \in A_\a$ for every $i < k$ and $\seq{B^i_n}{n \in \w} \in \mathrm{Seq}_\a$ for every $i < \ell$. Because the sequence $\vec{\hspace{.4mm}t}$ is is computable from these inputs, $\vec{\hspace{.4mm}t} \in \mathrm{Seq}_\a$. 
Furthermore, each $Y_n$ is a member of $\F$. 
Thus $\vec{\hspace{.4mm}t}$ is in the domain of $\psi$, and $\psi(\vec{\hspace{.4mm}t}\,) \in \G$.
By our definition of $\psi$ and of the $Y_n$, we have $\psi(\vec{\hspace{.4mm}t}\,) \sub F_i$ for every $i < k$ and $\psi(\vec{\hspace{.4mm}t}\,) \sub \psi(\vec s_i)$ for every $i < \ell$. Hence $\psi(\vec{\hspace{.4mm}t}\,) \in \G$ and
$$\psi(\vec{\hspace{.4mm}t}\,) \,\sub\, \textstyle \bigcap_{i < k}F_i \cap \bigcap_{i < \ell} \psi(\vec s_i).$$
Thus any finitely many members of $\G$ have a subset of their intersection in $\G$; in other words, $\G$ is a filter base.
\hfill {\scriptsize $\square\!\!$}
\vspace{2mm}

Returning to the proof of the theorem, we now produce an increasing sequence $\seq{\F_\g}{\g < \k}$ of filter bases via transfinite recursion. 
For the base case, let $\F_0$ be the Fr\'echet filter. 
If $\b$ is a limit ordinal, let $\F_\b = \bigcup_{\xi < \b}\F_\xi$.

If $\b$ is a successor ordinal, say $\b = \a+1$, then at stage $\b$ of the recursion we have an increasing sequence $\seq{\F_\xi}{\xi \leq \a}$ of filter bases. 
Let $\U_\a$ be some ultrafilter extending $\F_\a$, and let $\F$ be the (free) filter generated by the filter base $\U_\a \cap \mathrm{Set}_\a$. Note that $\F_\a \sub \F$. 
There is a function $f \in A_\b$ not dominated by $A_\a$. 
Let $\psi$ be the function described in our claim, as defined from $f$. 
Applying our claim,
$$\F_\b \,=\, \F \cup \set{\psi(\vec s\,)}{\vec s \in \mathrm{Seq}_\a \cap \F^\w}$$
is a filter base. 
As $\F_\a \sub \F$, we have $\F_\a \sub \F_\b$.
Furthermore, because $\psi(\vec s\,)$ is computable from $f$ and $\vec s$ for any given $\vec s \in A_\a$, $\psi(\vec \s,) \in \mathrm{Set}_\b$ whenever $\vec s \in \mathrm{Seq}_\a$. Hence $\F_\b \sub \mathrm{Set}_\b$. 
This completes the recursion. 

Let $\U = \bigcup_{\g < \k} \F_\g$. We claim that $\U$ is a $P$-point in $\w^*$.

To see that $\U$ is an ultrafilter, fix some $A \sub \w$. There is some $\a < \k$ such that $A \in \mathrm{Set}_\a$. 
At stage $\b = \a+1$ of our recursion, we chose an ultrafilter $\U_\a$ and then described a filter base $\F_\b$ extending $\U_\a \cap \mathrm{Set}_\a$. This implies that either $A \in \F_\b$ or else $\w \setminus A \in \F_\b$. Hence either $A \in \U$ or else $\w \setminus A \in \U$. As we already know $\U$ is a filter base, this means $\U$ is an ultrafilter.

To see that $\U$ is a $P$-point, suppose $\vec s = \seq{B_n}{n \in \w}$ is a sequence of sets in $\U$.
There is some $\a < \k$ with $\vec s \in \mathrm{Seq}_\a$. 
At stage $\b = \a+1$ of our recursion, we chose an ultrafilter $\U_\a$ and obtained a filter base $\F_\b$ extending $\U_\a \cap \mathrm{Set}_\a$. 
For each $n \in \w$, we have $B_n \in \mathrm{Set}_\a$, so because $\U_\a$ is an ultrafilter, either $B_n \in \U_\a$ or $\w \setminus B_n \in \U_\a$. But $\U \supseteq \F_\b \supseteq \U_\a \cap \mathrm{Set}_\a$, so in fact $B_n \in \U_\a$ for every $n$. Thus $\vec s = \seq{B_n}{n \in \w} \in \mathrm{Seq}_\a \cap (\U_\a \cap \mathrm{Set}_\a)^\w$, and at stage $\b$ of our recursion we added to $\F_\b$ a pseudo-intersection $\psi(\vec s\,)$ for the sequence $\vec s$.
Because $\F_\b \sub \U$, this shows that $\U$ contains a pseudo-intersection for $\vec s$.
\end{proof}

Ketonen proved a little more than just that $\dom = \continuum$ implies the existence of $P$-points; he showed that $\dom = \continuum$ implies every filter generated by $<\!\dom$ sets extends to a $P$-point. 
Let us point out that, with a little more work, the above argument can be adjusted to show that if there is a pathway of length $\k$, then every filter generated by $<\!\k$ sets extends to a $P$-point.

\begin{corollary}
It is consistent that \MH is false.
\end{corollary}

\noindent This follows from the previous two theorems, plus the fact that it is consistent there are no $P$-points in $\w^*$. 
This consistency result was first proved by Shelah (see \cite{Shelah,Wimmers}). 
Later work by Chodounsk\'y and Guzm\'an in \cite{CG} shows there are no $P$-points in the Silver model, or even in side-by-side Silver extensions where (unlike in Shelah's model) $\continuum$ may be arbitrarily large.

\begin{corollary}
$\Delta$ does not imply \MH.
\end{corollary}
\begin{proof}
In the model of Shelah without $P$-points, or in the Silver extensions studied by Chodounsk\'y and Guzm\'an, $\bdd = \dom = \aleph_1$. As mentioned in the previous section, $\bdd = \dom$ implies $\Delta$. So these are models of $\Delta$ and not $\MH$.
\end{proof}

Our next result gives yet another instance of the failure of \MH and the non-existence of pathways, this time in a model with $P$-points. 
Thus the converse of Theorem~\ref{thm:Ppoints} does not hold: the existence of pathways is not equivalent to the existence of $P$-points.

The Sacks forcing $\SS$ is the poset of all perfect subtrees of $2^{<\w}$, ordered by inclusion. (Recall that $T \sub 2^{<\w}$ is \emph{perfect} if every node in $T$ has two incompatible extensions.) The Sacks poset is an $\w^\w$-bounding, Axiom A (hence proper) forcing. (See \cite{BL} or \cite{GQ} for a reference.) For a given cardinal $\k$, let $\SS_\k$ denote the countable support product of $\k$ copies of $\SS$. Any model obtained by forcing with $\SS_\k$, for some regular $\k > \aleph_1$, over a model of \gch is called the side-by-side Sacks model with $\continuum = \k$.

\begin{theorem}
Let $\k$ be the successor of a regular uncountable cardinal. 
There are no pathways in the side-by-side Sacks model with $\continuum = \k$.
\end{theorem}
\begin{proof}
Let $\k$ be the successor of an uncountable regular cardinal, suppose $V \models \gch$, and let $G$ be a $\SS_\k$-generic filter over $V$. 
Aiming for a contradiction, suppose there is a pathway in $V[G]$. 

Recall that $\SS_\k \forces \dom = \aleph_1$, and that a pathway cannot have length $>\!\dom$. 
Thus all pathways in $V[G]$ have length $\w_1$. 
Furthermore, we claim there is a pathway $\seq{A_\a}{\a < \w_1}$ in $V[G]$ such that $A_\a$ does not dominate $A_{\a+1} \cap V$ for all $\a$. 
To this, let $\seq{A^0_\a}{\a < \w}$ be an arbitrary pathway in $V[G]$. Fix $\a < \w_1$. 
By the definition of a pathway, there is some $g \in A_{\a+1}^0$ not dominated by $A_\a^0$. 
Because $\SS_\k$ is $\w^\w$-bounding, there is some $h \in V \cap \w^\w$ with $h \not<^* g$. 
Furthermore, $h \in A_\b^0$ for some $\b$, and in fact we must have $\b > \a$ because the members of a pathway are increasing. Thus there is some $\b > \a$ such that $A_\b^0 \cap V$ is not dominated by $A_\a^0$. 
Therefore, by thinning out the sequence $\seq{A_\a^0}{\a < \w_1}$ appropriately, we can obtain a pathway $\seq{A_\a}{\a < \w_1}$ such that $A_\a$ does not dominate $A_{\a+1} \cap V$ for all $\a$. 

Fix a pathway $\< A_\a :\, \a < \w_1 \>$ in $V[G]$ such that $A_\a$ does not dominate $A_{\a+1} \cap V$ for all $\a < \w_1$. 
Also fix a corresponding sequence $\< \dot A_\a :\, \a < \w_1 \>$ of nice names in $V$. 

Like in the proof of Theorem~\ref{thm:Ppoints}, we wish to reason not only about the functions in some $A_\a$, but also about the things coded by functions in $A_\a$. 
For example, we can fix a computable bijection $\w \to 2^{<\w}$, and via this bijection, Sacks conditions $T \sub 2^{<\w}$ can be ``coded'' as subsets of $\w$, which in turn can be coded (via characteristic functions) as members of $\w^\w$. 
Likewise, a mapping between two subsets of $2^{<\w}$ can be coded as a function $\w \to \w$ in a canonical, computable way. 
Furthermore, because each $A_\a$ is closed under Turing reducibility, so are all these coded objects. For example, a subtree of $2^{<\w}$ is coded in $A_\a$ if it is computable from a bijection between two subsets of $2^{<\w}$ that is coded in $A_\a$. 
All of this will be used without further comment in what follows.

Let $\dot x_\g$ be a (nice) name for the $\SS$-generic real added by the $\g^{\mathrm{th}}$ coordinate of $\SS_\k$. 
In $V[G]$, every subset of $2^{<\w}$ is canonically coded as a function, and in particular the real $x_\g = (\dot x_\g)_G$ (which is naturally identified with a branch through $2^{<\w}$) has a code appearing in some $A_\a$. 
Thus, in $V$, there is for each $\g < \k$ some $p_\g \in \SS_\k$ and $\a_\g < \w_1$ such that $p_\g \forces \dot x_\g$ is coded in $\dot A_{\a_\g}$. 

Working in the ground model, we have $|\SS| = \aleph_1$ (by \ch). 
Because $\k$ is regular and $>\!\aleph_1$, there is some particular $T \in \SS$ and a stationary $S \sub \k$ such that $p_\g(\g) = T$ for all $\g \in S$. 
By the same reasoning, we may (and do) assume, by thinning out $S$ if necessary, that there is some particular $\a < \w_1$ with $\a_\g = \a$ for all $\g \in S$. 
Using the generalized $\Delta$-system lemma (which applies because \gch holds and $\k$ is not the successor of a singular cardinal), we may (and do) assume, by thinning out $S$ again if needed, that there is some $\bar p \in \SS_\k$ such that $p_\g \rest \g = \bar p$ for all $\g \in S$, and $(\mathrm{supp}(p_\g) \cap \mathrm{supp}(p_\dlt)) = \mathrm{supp}(\bar p)$ for all $\g,\dlt \in S$ with $\g \neq \dlt$. (In other words, the supports of the conditions $p_\g$, $\g \in S$, form a generalized $\Delta$-system of countable sets, with $\mathrm{supp}(\bar p)$ being the root of the $\Delta$-system.)

To summarize: we have a stationary $S \sub \k$, $T \in \SS$, $\a < \w_1$, and $\bar p \in \SS_\k$ such that, for all $\g \in S$,  $p_\g \forces \dot x_\g$ is coded in $A_\a$, $p_\g(\g) = T$, $p_\g \rest \g = \bar p$, and if $\g \neq \dlt \in S$ then $(\mathrm{supp}(p_\g) \cap \mathrm{supp}(p_\dlt)) = \mathrm{supp}(\bar p)$.

Let $B \sub T$ denote the set of all branching nodes of $T$: that is, $B = \set{t \in T}{ t \cat 0, t \cat 1 \in T }$. 
Fix an order-isomorphism $\varphi: 2^{<\w} \to B$. (Because $T$ is a perfect tree, $B$ is in fact order-isomorphic to $2^{<\w}$.) 
Because $\varphi$ can be canonically coded as a function in $\w^\w$, there is some $\SS_\k$-condition $\bar q \leq \bar p$ and some $\b \geq \a$ such that $\bar q \forces$ $\varphi$ is coded in $\dot A_\b$. 

Recall that, in $V[G]$, $A_\b$ does not dominate $A_{\b+1} \cap V$.  
Thus, in $V$, there is some condition $\bar r \leq \bar q$ and some function $f \in \w^\w$ such that $\bar r \forces f \in \dot A_{\b+1}$ and $f$ is not dominated by any function in $\dot A_\b$.

Because $\set{\mathrm{supp}(p_\g)}{\g \in S}$ is a (generalized) $\Delta$-system and $\mathrm{supp}(\bar r)$ is countable, $\mathrm{supp}(p_\g) \cap \mathrm{supp}(\bar r) = \mathrm{supp}(\bar p)$ for all but countably many $\g \in S$. Thinning out $S$ one last time, let us suppose $\mathrm{supp}(p_\g) \cap \mathrm{supp}(\bar r) = \mathrm{supp}(\bar p)$ for all $\g \in S$. 
Note that this implies $p_\g$ and $\bar r$ have a common extension (namely $(p_\g \setminus p_\g \rest \mathrm{supp}(\bar p)) \cup \bar r$) for all $\g \in S$.

Fix $\g \in S$. Let $\dot y_\g$ be a nice name for the function $\varphi^{-1} \circ x_\g$ in $V[G]$. In other words, $\dot y_\g$ is a name for an element of $2^\w$ which reveals, via $\varphi$, the way in which the $\SS$-generic real $x_\g$ traces on $B$. 
Any common extension of $\bar r$ and $p_\g$ forces that all of $\dot x_\g,\varphi,\varphi^{-1},\dot y_\g$ are coded in $\dot A_\b$. (For $\dot x_\g$, this is true because $p_\g \forces$ $\dot x_\g$ is coded in $\dot A_\a$ and $\dot A_\a \sub \dot A_\b$; for $\varphi$ this is true because $\bar r \forces$ $\varphi$ is coded in $\dot A_\b$; for $\varphi^{-1}$ and $\dot y_\g$, this is true because, in the extension, (codes for) $\varphi^{-1}$ and $y_\g$ are computable from (codes for) $\varphi$ and $x_\g$.) 

In $V[G]$, let $I_\g = \set{j}{y_\g(j) = 1} = y_\g^{-1}(1)$ and let $h_\g(n)$ denote the $n^\mathrm{th}$ element of $I_\g$. 
Let $\dot I_\g$ and $\dot h_\g(n)$ be nice names for these two objects in $V$. 
Because $I_\g$ and $h_\g$ are computable from $y_\g$, any common extension of $\bar r$ and $p_\g$ forces $\dot I_\g,\dot h_\g \in$ are coded in $A_\b$ (as well as $\dot x_\g,\varphi,\varphi^{-1},\dot y_\g$). 

Given $C \in [\w]^\w$, let 
$T_C \,=\, \set{t \in 2^{<\w}}{t^{-1}(1) \sub C}.$ 
In other words, $T_C$ is the tree that branches to $0$ at every node, and branches also to $1$ at (and only at) levels in $C$. 
Let $\varphi[T_C] \!\downarrow$ denote the downward closure of the image of $T_C$ under $\varphi$. (So, for example, $T_\w = 2^{<\w}$ and $\varphi[T_\w] \!\downarrow \ = B\!\downarrow\ = T$.) Note that $\varphi[T_C] \!\downarrow\, \in \SS$ and $\varphi[T_C] \!\downarrow\ \leq T$ (as $\SS$-conditions).

Because $f \in V$, there is some infinite $C \sub \w$ with $C \in V$ such that the $n^\mathrm{th}$ element of $C$ is $\geq\!f(n)$.  
For any $\g \in S$, observe that
$$s \,=\, (p_\g \setminus \bar p) \cup \bar r \cup (\g,\varphi[T_C]\!\downarrow)$$
is a common extension of $p_\g$ and $\bar r$. 
Because $s(\g) = \varphi[T_C]\!\downarrow\,$, $s$ forces that $\dot x_\g$ is a branch through the tree $\varphi[T_A]\!\downarrow\,$. But for any branch $b$ through $\varphi[T_C]\!\downarrow\,$, we can have $\varphi^{-1} \circ b(n) = 1$ only if $n \in C$ (by the definition of $\varphi[T_C]\!\downarrow\,$). Thus $s \forces \dot I_\g \sub C$. But if $I_\g \sub C$, then the $n^\mathrm{th}$ element of $I_\g$ is even bigger than the $n^\mathrm{th}$ element of $C$, which in turn is $\geq\!f_{\b+1}(n)$. Thus, in $V[G]$, the function $h_\g$ enumerating $I_\g$ dominates $f$. Hence $s \forces \dot h_\g \geq f$.
But $f$ is unbounded over $A_\b$ in $V[G]$, and $r \forces \dot h_\g \in \dot A_\b$. Contradiction!
\end{proof}

Let us note in passing that, with a little more work, the above proof can be modified to give the same conclusion in the iterated Sacks model. 

\begin{corollary}
The existence of $P$-points does not imply the existence of pathways.
\end{corollary}
\begin{proof}
By the previous theorem, it suffices to note that there are $P$-points in the side-by-side Sacks models. This was proved by Laver in \cite{Laver}. 
\end{proof}

Interestingly, Laver's argument in \cite{Laver} does not show that all $P$-points from the ground model are preserved by $\SS_\k$ (although this is true for the iterated Sacks poset). Instead, Laver constructs specific $P$-points in the ground model that are preserved by $\SS_\k$. It is an open question whether all ground model $P$-points generate $P$-points in side-by-side Sacks models. 

Note that neither of the results from this section addresses the question of whether it is consistent for $\Delta$ to fail. In fact, $\bdd = \dom = \aleph_1$ in every known model without $P$-points, and in the side-by-side and iterated Sacks models. Because $\bdd = \dom$ implies $\Delta$, all the models considered in this section satisfy $\Delta$.

\begin{question}[Roitman]
Is $\Delta$ a theorem of \zfc?
\end{question}

\noindent Because both $\bdd = \dom$ and $\dom = \continuum$ imply $\Delta$, any model in which $\Delta$ fails must satisfy $\bdd < \dom < \continuum$, and in particular it must satisfy $\continuum \geq \aleph_3$. On the one hand, this means that a countable support iteration of proper posets is not useful for solving the problem. On the other hand, the results in the next section make it seem doubtful that a ccc poset could be useful either.

\section{Finding pathways in forcing extensions}\label{sec:ccc}

Roitman proved in \cite{Roitman2} that $\Delta$ implies $\nabla (\w+1)^\w$ is paracompact. As mentioned in Section 2, later work of the first author and Gartside in \cite{BAG} shows that $\Delta$ is in fact equivalent to the paracompactness of $\nabla (\w+1)^\w$.

Roitman also proved in \cite{Roitman1} that following any ccc forcing iteration with uncountable cofinality, $\nabla (\w+1)^\w$ is paracompact (even more: she showed $\nabla_{n \in \w}X_n$ is paracompact whenever each $X_n$ is compact). As we now know $\Delta$ is equivalent to the paracompactness of $\nabla (\w+1)^\w$, this means that $\Delta$ holds in such forcing extensions. 

In this section we strengthen Roitman's result in two ways: by strengthening the conclusion from $\Delta$ to the existence of pathways, and by extending the class of posets for which this conclusion holds. 

Given a ccc poset $\PP$, let $\MH(\PP)$ denote the following statement:
\begin{itemize}
\item[$\MH(\PP)\!: \hspace{-1.75mm}$] \hspace{1mm} For some uncountable cardinal $\k$, there is an increasing sequence $\seq{M_\a}{\a < \k}$ of transitive models of $\zfc^-$ such that $\PP \sub \bigcup_{\a < \k}M_\a$, 
$\bigcup_{\a < \k}M_\a$ is countably closed, and for all $\a < \k$,
\begin{itemize}
\item[$\circ$] $M_{\a+1} \cap \PP$ is not dominated by $M_\a \cap \w^\w$, 
\item[$\circ$] $\PP_\a = M_\a \cap \PP \in M_\a$, and $\PP_\a \embeds \PP$, and
\item[$\circ$] $M_\a$ witnesses that $M_\a \cap \PP = \PP_\a$ is $\w^\w$-bounding, in the sense that for every nice name $\dot f$ for a function in $\w^\w$ with $\dot f \in M_\a$, there is some $g \in M_\a \cap \w^\w$ such that $\forces_{\PP_\a} \dot f <^* g$.
\end{itemize}
\end{itemize}

Note that $\MH(\PP)$ implies $\PP$ is $\w^\w$-bounding. This is because if $\dot f$ is a (nice) name for a function in $\w^\w$, then because $\PP$ is ccc, $\PP \sub \bigcup_{\a < \k}M_\a$, and $\bigcup_{\a < \k}M_\a$ is countably closed, we get $\dot f \in \bigcup_{\a < \k}M_\a$. Because $M_\a \cap \PP \embeds \PP$ and 
$M_\a$ witnesses that $M_\a \cap \PP$ is $\w^\w$-bounding, this implies that the evaluation of $\dot f$ in an extension will be dominated by some function in the ground model. 
Hence $\MH(\PP)$ may be thought of as a strong version of $\w^\w$-boundedness. 

The insistence that the $M_\a$ be transitive is not strictly necessary. It is fine, for example, if the $M_\a$ are elementary submodels of some $H(\theta)$. (If so, then identifying the members of $\PP$ with ordinals $<\!\mu = |\PP|$ and replacing each $M_\a$ with its transitive collapse does not change any other aspect of the definition. In this sense, an elementary submodels version of the definition implies the stated version.) What is really needed is that the $M_\a$ agree with $V$ on what $\w$ and $\w^\w$ are, which can fail in non-standard models.

\begin{theorem}\label{thm:AlmostAllCCC}
If $\PP$ is a ccc poset and $\MH(\PP)$ holds, then $\PP$ forces that pathways exist.
\end{theorem}
\begin{proof}
Let $\PP$ be a ccc poset and suppose $\MH(\PP)$ holds in the ground model $V$. 
Let $G$ be a $\PP$-generic filter over $V$. 

Fix a sequence $\seq{M_\a}{\a < \k}$ in $V$ witnessing $\MH(\PP)$. 
For each $\a < \k$, let
$$A_\a \,=\, \set{(\dot f)_G}{\dot f \in M_\a \text{ and $\dot f$ is a nice $\PP$-name for a function $\w \to \w$}}.$$
We claim that $\seq{A_\a}{\a < \k}$ is a pathway in $V[G]$. 

As the sequence $\seq{M_\a}{\a < \k}$ is increasing, $\seq{A_\a}{\a < \k}$ is too. 

If $f \in \w^\w$ in $V[G]$, then there is a nice name $\dot f$ for $f$ in $V$. 
Because $\PP$ is ccc, $\dot f$ consists of countably many pairs of the form $((m,n),p)$ with $p \in \PP$. 
Because $\PP \sub \bigcup_{\a < \k}M_\a$ and $\bigcup_{\a < \k}M_\a$ is countably closed, $\dot f \in \bigcup_{\a < \k}M_\a$. This implies $f \in A_\a$ for some $\a < \k$. 
As $f$ was arbitrary, $\bigcup_{\a < \k}A_\a = \w^\w$.

Next, fix $\a < \k$. 
Because $M_{\a+1} \cap \w^\w$ is not dominated by $M_\a \cap \w^\w$, there is some $g \in M_{\a+1} \cap \w^\w$ such that $g \not<^* h$ for all $h \in M_\a \cap \w^\w$.
If $f \in A_\a$, there is some nice name $\dot f \in M_\a$ with $(\dot f)_G = f$. 
Because $M_\a$ witnesses that $\PP_\a$ is $\w^\w$-bounding, there is some $h \in M_\a \cap \w^\w$ such that $\forces_{\PP_\a} \dot f <^* h$. 
Because $\PP_\a \embeds \PP$, this means $\forces_\PP \dot f <^* h$. Thus in the extension, $f <^* h$, which implies $g \not<^* f$. As $f$ was an arbitrary member of $A_\a$, this means $g$ witnesses that $A_{\a+1}$ is not dominated by $A_\a$. 

Now suppose $f,g \in A_\a$. This means there are nice names $\dot f$ and $\dot g$ for $f$ and $g$ in $M_\a$. But then
$$\set{((2i,j),p)}{((i,j),p) \in \dot f} \cup \set{((2i-1,j),p)}{\vphantom{\dot f}((i,j),p) \in \dot g}$$
is a nice name for $f \vee g$, and this name is in $M_\a$ because it is definable from $\dot f$ and $\dot g$ and $M_\a \models \zfc^-$. 
Hence $f \vee g \in A_\a$. 

Finally, suppose $g \in A_\a$ and $f \in \w^\w$ is Turing reducible to $g$. 
This means $f$ is computable from $g$, in the sense that there is an oracle Turing machine $T$ that, when using $g$ as an oracle, outputs $f(n)$ on input $n$ for all $n \in \w$. 
For any given $n$, there is some $h(n)$ large enough that $T$ computes $f(n)$ in $\leq\! h(n)$ steps. 
In particular, $T$ does not read more than the first $h(n)$ values in $g$, and if $g'$ is any function with $g \rest h(n) = g' \rest h(n)$, then $T$ will correctly compute $f(n)$ using $g'$ as an oracle instead of $g$. 
For any given finite sequence $\seq{k_i}{i \leq n}$ of natural numbers, define a function $G_{\seq{k_i}{i \leq n}}: \w \to \w$ by 
$$G_{\seq{k_i}{i \leq n}}(i) \,=\, \begin{cases}
k_i &\text{ if } i \leq n \\
0 &\text{ if } i > n.
\end{cases}$$
Because $M_\a$ witnesses that $\PP_\a$ is $\w^\w$-bounding, there is some ground model function $\bar h \in M_\a \cap \w^\w$ such that $h(n) \leq \bar h(n)$ for all $n$. 
Thus, letting $\dot g \in M_\a$ be some nice name for $g$,  
\begin{align*}
\Big\{ \, ((n,m),p) \,:\ \, &p \in \PP_\a \text{ decides $\seq{\dot g(i)}{i < \bar h(n)} = \seq{k_i}{i < \bar h(n)}$, and } \\
& T \text{ outputs } m \text{ with input } n \text{ using $G_{\seq{k_i}{i < \bar h(n)}}$ as an oracle} \, \Big\}
\end{align*}
is a name for $f$. 
(This uses the fact that $\PP_\a \embeds \PP$, which implies that the members of $\PP_\a$ that decide $\seq{\dot g(i)}{i < \bar h(n)}$ form a dense subset of $\PP$.) 
This name is in $M_\a$, because it is definable (as above) from $\dot g$, $\bar h$, $T$, and $\PP$, all of which are in $M_\a$. 
Because $M_\a \models \zfc^-$, there is also a nice name for $f$ in $M_\a$. 
Hence $f \in A_\a$, and this shows $A_\a$ is closed under Turing reducibility.
\end{proof}

It is unclear whether the existence of pathways can be improved to \MH in the conclusion of this theorem. The problem is that \MH requires the $A_\a$ to be closed under set-theoretic definability, not merely Turing reducibility. Our proof explicitly relies on the fact that if $f \leq_T g$, then any given entry of $f$ can be computed by knowing some finitely many entries of $g$. The same idea does not work for set-theoretic definability, where some of the entries of $f$ might depend somehow on infinitely many entries of $g$. 

\begin{question}
If $\PP$ is a ccc poset and $\MH(\PP)$ holds, does $\PP$ forces \MH?
\end{question}

Of course, Theorem~\ref{thm:AlmostAllCCC} raises the question: Under what circumstances, and for which posets $\PP$, does $\MH(\PP)$ hold? 
We shall be particularly (but not exclusively) interested in this question when $\PP$ is the measure algebra of weight $\mu$, the standard poset for adding $\mu$ random reals simultaneously. 
For technical reasons (that become clear in the proof of Theorem~\ref{thm:RandomPoset}), we define the members of $\BB_\mu$ to be Borel codes, rather than Borel sets or equivalence classes of Borel sets.

\begin{lemma}\label{lem:RandomLemma}
For any cardinal $\mu$, $\MH(\BB_\mu)$ if, 
for some uncountable cardinal $\k$, there is an increasing sequence $\seq{M_\a}{\a < \k}$ of transitive models of $\zfc^-$ such that $\mu \in M_0$, 
$\BB_\mu \sub \bigcup_{\a < \k}M_\a$, 
$\bigcup_{\a < \k}M_\a$ is countably closed, and 
$M_{\a+1} \cap \w^\w$ is not dominated by $M_\a \cap \w^\w$ 
for all $\a < \k$. 
In other words, the last two items in the definition of $\MH(\BB_\mu)$ are automatic, in the sense that when $\PP = \BB_\mu$, they follow already from the previous conditions (plus the condition that $\mu \in M_0$).
\end{lemma}
\begin{proof}
Suppose $\seq{M_\a}{\a < \k}$ is a sequence of models having the properties stated in the lemma. 
We must check that $M_\a \cap \BB_\mu \in M_\a$, that $M_\a \cap \BB_\mu \embeds \BB_\mu$, and that $M_\a \cap \BB_\mu$ witnesses that $M_\a \cap \BB_\mu$ is $\w^\w$-bounding. 

Fix $\a < \k$. 
Because $\mu \in M_0 \sub M_\a$, we have $\mu \in M_\a$. Thus $\BB_\mu$ is definable as the set of all Borel codes for members of $2^\mu$, and this definition is absolute for transitive models containing $\mu$. Hence $M_\a \cap \BB_\mu \in M_\a$.

To see that $M_\a \cap \BB_\mu \embeds \BB_\mu$, suppose $A$ is a maximal antichain in $M_\a \cap \BB_\mu$. This means that $A$ consists of codes for non-null Borel sets, with the codes in $M_\a$, such that (1) any two of these Borel sets intersect in a null set, and (2) if any other Borel set coded in $M_\a$ has null intersection with all these Borel sets, it is null. 
Because $M_\a \models \zfc^-$, and because basic facts about Borel codes (like the measure of the set they code) are absolute between models of $\zfc^-$, condition (2) is equivalent to ($2'$) the sum of the measures of these Borel sets is equal to $1$. 
But $(1)$ and $(2')$ together imply that $A$ is a maximal antichain in $\BB_\mu$, not just in $M_\a \cap \BB_\mu$. 

To see that $M_\a$ witnesses that $M_\a \cap \BB_\mu$ is $\w^\w$-bounding, let $\dot f$ be a nice $(M_\a \cap \BB_\mu)$-name for a function in $\w^\w$. 
That is, 
$$\dot f = \set{((m,n),p_{m,n})}{m,n \in \w},$$ 
where each $p_{m,n} \in M_\a \cap \BB_\mu$. 
Going through the usual proof that $\BB_\mu$ is $\w^\w$-bounding, we obtain a ground model function
$$g(n) \,=\, \min \set{k}{\textstyle \sum_{i < k} \lambda(p_{n,i}) > 1 - \nicefrac{1}{4^n}}$$
(where $\lambda$ denotes the Lebesgue measure) such that $\forces_{\BB_\mu} \dot f <^* g$. 
But this function $g$ is definable from $\dot f$, so $g \in M_\a$. Moreover, the usual proof that $\forces_{\BB_\mu} \dot f <^* g$ can be carried out in $M_\a$, and this adaptation of the proof shows that $\forces_{M_\a \cap \BB_\mu} \dot f <^* g$. 
\end{proof}

Note that if $\mu$ is a definable cardinal with a ``nice'' definition that is absolute to transitive models (e.g., if $\mu = \aleph_n$ for some $n$, or $\mu = \aleph_{\w_1+\w^2+49}$), then the condition $\mu \in M_0$ is superfluous, and the statement of the lemma can be strengthened from ``if'' to ``if and only if''.

\begin{theorem}\label{thm:RandomPoset}
Suppose $\PP$ is a forcing iteration (with any support) of length $\lambda$, where $\mathrm{cf}(\lambda) > \w$. Furthermore, suppose $\PP$ adds unbounded reals at cofinally many stages of the iteration. In any forcing extension by $\PP$, if $\BB_\mu$ denotes the measure algebra of weight $\mu$, then $\MH(\BB_\mu)$ holds.
\end{theorem}
\begin{proof}
Let $\PP$ be a forcing iteration as described in the statement of the theorem, and let $G$ be a generic for $\PP$. 
For each $\g < \lambda$, let $G_\g$ denote the restriction of $G$ to only the first $\g$ stages of the iteration.  

Let $\k = \mathrm{cf}(\lambda)$, and by recursion obtain a $\k$-sequence $\seq{\g_\a}{\a < \k}$ of ordinals $<\!\lambda$ such that $\seq{\g_\a}{\a < \k}$ is cofinal in $\lambda$, and for all $\a < \k$ there is an unbounded real added at stage $\xi$ of the iteration for some $\xi \in (\g_\a,\g_{\a+1}]$. 
In $V[G]$, define $M_\a = H(\mu^+)^{V[G_{\g_\a}]}$ for all $\a < \k$. 
We claim $\seq{M_\a}{\a < \k}$ witnesses $\MH(\BB_\mu)$.

Clearly $\mu \in M_0$. The sequence of $M_\a$'s is increasing because the sequence of $V[G_{\g_\a}]$'s is increasing.  
Because each $V[G_{\g_\a}]$ is a model of \zfc, we have $V[G_{\g_\a}] \models (H(\mu^+) \models \zfc^-)$. By the absoluteness of the satisfaction relation, $M_\a \models \zfc^-$ for all $\a$. 
Every code for a Borel set in $\BB_\mu$ (depending on one's choice of coding) is essentially a countable partial function $\mu \times \w \to \w$, and every such countable partial function appears in $V[G_\g]$ for some $\g < \k$ (because $\mathrm{cf}(\k) > \w$). Thus $\BB_\mu \sub \bigcup_{\a < \k}M_\a$. 
Similar reasoning shows that $\bigcup_{\a < \k}M_\a$ is countably closed. 
Finally, $M_{\a+1} \cap \w^\w$ is not dominated by $M_\a \cap \w^\w$ because an unbounded real is added at stage $\xi$ of the iteration for some $\xi \in (\g_\a,\g_{\a+1}]$. 
That 
$\MH(\BB_\mu)$ holds now follows from Lemma~\ref{lem:RandomLemma}. 
\end{proof}

By convention, ``the random model'' means any model obtained from a model of \ch after adding $\aleph_2$ random reals. But the ``the'' in this name is misleading, because some properties of the forcing extension may depend on precisely which model of \ch we started with. 
An unpublished result of Kunen shows that if we begin with a model of \ch, then force to add $\aleph_1$ Cohen reals, and then force to add $\geq\!\aleph_2$ random reals, then we get a model with $P$-points. 
Thus it is consistent that ``the'' random model contains $P$-points. 
(Another result along these lines was obtained by the third author in \cite{Dow}: adding $\aleph_2$ random reals to a model of $\ch+\square_{\w_1}$ gives a model with $P$-points.) 
However, as mentioned in the introduction, it is an important open problem whether adding random reals to \emph{any} model of \ch produces a model with $P$-points. 
The following corollary to the previous two theorems contains Kunen's result as a special case, and shows that in fact many  forcings can be used to produce models $V$ of \ch such that ``the'' random model built from $V$ contains $P$-points.

\begin{corollary}
Suppose $\PP$ is a forcing iteration (with any support) of length $\lambda$, where $\mathrm{cf}(\lambda) > \w$, and suppose $\PP$ adds unbounded reals at cofinally many stages of the iteration. 
Then forcing with $\PP * \dot \BB_\mu$ produces a model with pathways (consequently, a model where $\Delta$ holds and $P$-points exist).
\end{corollary}
\begin{proof}
This follows directly from Theorems~\ref{thm:AlmostAllCCC} and \ref{thm:RandomPoset}.
\end{proof}

Note that every finite support iteration of nontrivial forcings (of the appropriate length) satisfies the hypotheses of this corollary. This is because finite support iterations of nontrivial forcings add Cohen reals at limit stages of countable cofinality.

\begin{question}
Does \ch imply $\MH(\BB_{\w_2})$?
\end{question}

\noindent A positive answer to this question would imply that there are $P$-points in the random model (all of them). 

In \cite{RW}, Roitman and Williams ask whether $\Delta$ holds in the model obtained by adding $\aleph_3$ random reals to the $\aleph_2$-Cohen model, i.e., the model obtained by forcing with $\CC_{\w_2} * \BB_{\w_3}$ over a model of \ch. 
This was seen to be the simplest model for which it was unknown whether $\Delta$ holds (in part because $\bdd = \aleph_1 < \dom = \aleph_2 < \continuum = \aleph_3$ in this model). 

\begin{corollary}
$\Delta$ holds in any model obtained by forcing with $\CC_{\w_2} * \BB_{\w_3}$.
\end{corollary}
\begin{proof}
One may view $\CC_{\w_2}$ as a length-$\w_2$ finite support iteration that adds Cohen reals at every stage.
\end{proof}

Lastly, let us consider an arbitrary ccc poset $\PP$. 

\begin{question}\label{q:a}
Suppose \ch holds in the ground model and $\forces_\PP \dom < \aleph_\w$. Does $\PP$ forces that $\Delta$ holds?
\end{question}

To approach this question, let us suppose $G$ be a $\PP$-generic filter over $V$. Let $\dom = \dom^{V[G]}$, the dominating number of the extension. (The dominating number of the ground model is $\aleph_1$, because \ch holds). 
If $\dom = \aleph_1$ then $\bdd = \dom$ in $V[G]$, and consequently $\Delta$ holds. Thus in what follows, we may (and do) assume $\dom > \aleph_1$. 
Let $D = \< \dot f_\a :\, \a < \dom \>$ be a sequence of $V$-names for members of $\w^\w$ such that $\mathbf{1}_\BB \forces D$ is a dominating family.

In $V$, we can find an increasing sequence $\seq{M_\a}{\a < \dom}$ of elementary submodels of $H(\theta)$ (where $\theta$ is a sufficiently large regular cardinal) such that 
\begin{itemize}
\item[$\circ$] each $M_\a$ is countably closed (in $V$), with $\card{M_\a} < \dom$, 
\item[$\circ$] if $\mathrm{cf}(\a) > \w$, then $M_\a = \bigcup_{\xi < \a}M_\xi$, and 
\item[$\circ$] $D \in M_\a$ for all $\a < \dom$.
\end{itemize}
Note that $\dom$ has uncountable cofinality, which implies $M = \bigcup_{\a < \dom}M_\a$ is a countably closed elementary submodel of $H(\theta)$ with $\card{M} = \dom$.

Let $\PP_M = \PP \cap M$. Because $\PP$ is ccc and $M$ is countably closed, $\PP_M \embeds \PP$. Thus there is a $\PP_M$-name $\dot \Q = \PP / \PP_M$ for a Boolean algebra such that $\PP = \PP_M * \dot \Q$. Furthermore, $\forces_{\PP_M} \dot \Q$ is $\w^\w$-bounding.

In other words, we are able to factor $\PP$ into two pieces: one part $\PP_M$ that adds unbounded reals, and another part $\dot \Q$ such that $\forces_{\PP_M} \dot \Q$ is $\w^\w$-bounding. 
Now the real question behind Question~\ref{q:a} is whether $\forces_{\PP_M} \MH(\dot \Q)$.

If $\forces_{\PP_M} \MH(\dot \Q)$, then by Theorem~\ref{thm:AlmostAllCCC}, $\PP = \PP_M * \dot \Q$ forces that pathways exist, and therefore $\Delta$ holds. 
To see that this may be plausible, consider  
the posets $\PP_\a = \PP \cap M_\a$ for each $\a < \dom$, which are all completely embedded in $\PP$. 
These $\PP_\a$'s act like the initial stages of a forcing iteration, and something like the proof of Theorem~\ref{thm:RandomPoset} can then be used to show that if $\forces_{\PP_M} \dot \Q = \dot \BB_\mu$, then $\forces_{\PP_M} \MH(\dot \Q)$. 
The problem, however, is that if $\dot \Q$ is a badly behaved poset, the intermediate models arising from these $\PP_\a$ do not seem to give a witness to $\MH(\dot \Q)$ after forcing with $\PP_M$. What kind of structure do these intermediate models impose on $\dot \Q$, and is this structure enough to deduce $\Delta$ (or even pathways)? We do not yet know, and so Question~\ref{q:a} remains open for now.




\begin{thebibliography}{99}

\bibitem{BAG} H. Barriga-Acosta and P. M. Gartside, ``Monotone normality and nabla products,'' \emph{Fundamenta Mathematicae} \textbf{254} (2021), pp. 99--120.
\bibitem{BL} J. E. Baumgartner and R. Laver, ``Iterated perfect-set forcing,'' \emph{Annals of Mathematical Logic} \textbf{17} (1979), pp. 271--288.
\bibitem{Blass} A. Blass, ``Combinatorial cardinal characteristics of the continuum,'' in \emph{Handbook of Set Theory,} eds. M. Foreman and A. Kanamori, Springer-Verlag (2010), pp. 395--489. 
\bibitem{CG} D. Chodounsk\'y and O. Guzm\'an, ``There are no $P$-points in Silver extensions,'' \emph{Israel Journal of Mathematics} \textbf{232} (2019), pp. 759--773.
\bibitem{Cohen} Paul E. Cohen, ``$P$-points in random universes,'' \emph{Proceedings of the American Mathematical Society} \textbf{74} (1979), pp. 318--321.
\bibitem{Dow} A. Dow, ``$P$-filters and Cohen, random, and Laver forcing,'' \emph{Topology and Its Applications} \textbf{281} (2020), article no. 107200.
\bibitem{vanDouwen} E. K. van Douwen, ``Covering and separation properties of box products,'' \emph{Surveys in General Topology}, Elsevier (1980), pp. 55--129.
\bibitem{FB} D. J. Fern\'andez-Bret\'on, ``Generalized pathways,'' unpublished manuscript available at \texttt{https://arxiv.org/abs/1810.06093}.
\bibitem{FBH} D. J. Fern\'andez-Bret\'on and M. Hru\v{s}\'ak, ``Gruff ultrafilters,'' \emph{Topology and Its Applications} \textbf{210} (2016), pp. 355--365.
\bibitem{GQ} S. Geschke and S. Quickert, ``On Sacks forcing and the Sacks property,'' in \emph{Classical and New Paradigms of Computation and their Complexity Hierarchies}, eds. B. L\"{o}we et al., Trends in Logic \textbf{23} (2007), pp. 95--139.
\bibitem{Ketonen} J. Ketonen, ``On the existence of $P$-points in the Stone-\v{C}ech compactification of integers,'' \emph{Fundamenta Mathematicae} \textbf{92} (1976), pp. 91--94.
\bibitem{Kunen} K. Kunen, ``Paracompactness of box products of compact spaces,'' \emph{Transactions of the American Mathematical Society} \textbf{240} (1978), pp. 307--316.
\bibitem{Laver} R. Laver, ``Products of infinitely many perfect trees,'' \emph{Journal of the London Mathematical Society} \textbf{29} (1984), pp. 385--396.
\bibitem{NZ} I. Neeman and J. Zapletal, ``Proper forcings and absoluteness in $\mathrm{L}(\R)$,'' \emph{Commentationes Mathematicae Universitatis Carolinae} \textbf{39} (1998), pp. 281--301.
\bibitem{Roitman0} J. Roitman, ``Paracompact box products in forcing extensions,'' \emph{Fundamenta Mathematicae} \textbf{102} (1979), pp. 219--228.
\bibitem{Roitman1} J. Roitman, ``More paracompact box products,'' \emph{Proceedings of the American Mathematical Society} \textbf{74} (1979), pp. 171--176.
\bibitem{Roitman2} J. Roitman, ``Paracompactness and normality in box products: old and new,'' in \emph{Set Theory and its Applications} (2011), ed. L. Babinkostova et al., Contemporary Mathematics 533, Providence, RI, pp. 157--181.
\bibitem{RW} J. Roitman and S. Williams, ``Paracompactness, normality, and related properties of topologies on infinite products,'' \emph{Topology and its Applications} \textbf{195} (2015), pp. 79--92.
\bibitem{Rudin} M. E. Rudin, ``The box product of countably many compact metric spaces,'' \emph{General Topology and its Applications} \textbf{2} (1972), pp. 293--298.
\bibitem{Shelah} S. Shelah, \emph{Proper and Improper Forcing}, $2^\mathrm{nd}$ ed., Perspectives in
Mathematical Logic (1998), Springer-Verlag, Berlin.
\bibitem{Williams} S. Williams, ``Box products,'' in \emph{Handbook of Set-Theoretic Topology} (1984), eds. K. Kunen and J. E. Vaughan, North-Holland, pp. 169--200.
\bibitem{Wimmers} E. L. Wimmers, ``The Shelah $P$-point independence theorem,'' \emph{Israel Journal of Mathematics} \textbf{43} (1982), pp. 28--48.

\end{thebibliography}
\end{document}